\documentclass[12pt]{article}
\usepackage[latin1]{inputenc}
\usepackage{amsmath,amsthm}
\usepackage{amsfonts,mathrsfs}
\usepackage{color}
\usepackage[colorlinks,anchorcolor=blue,citecolor=blue]{hyperref}
\allowdisplaybreaks
\theoremstyle{plain}
\newtheorem{thm}{\noindent\bf Theorem}[section]
\newtheorem{cor}{\noindent\bf Corollary}[section]
\newtheorem{lem}{\noindent\bf Lemma}[section]
\newtheorem{defn}{\noindent\bf Definition}[section]
\newtheorem{prop}{\noindent\bf Proposition}[section]
\theoremstyle{remark}
\newtheorem{rmk}{\noindent \bf \textit{Remark}}[section]

\newcommand{\D}{\displaystyle}
\newcommand{\non}{\nonumber}
\newcommand{\Pm}{\mathbb{P}}
\newcommand{\Em}{\mathbb{E}}

\usepackage{ulem}
\newcounter{num}

\newcommand{\wo}{W^{(\omega)}}
\newcommand{\zo}{Z^{(\omega)}}
\newcommand{\uo}{U^{(\omega)}}
\newcommand{\hoe}{H_{\varepsilon}}
\newcommand{\woe}{W^{(\omega+\omega_{\varepsilon})}}
\newcommand{\zoe}{Z^{(\omega+\omega_{\varepsilon})}}
\newcommand{\wqe}{W^{(q+\omega_{\varepsilon})}}

\newcommand{\hoap}{\mathsf{h}^{(\omega;p)}_{(a)}}
\newcommand{\woap}{\mathsf{W}^{(\omega;p)}_{(a)}}
\newcommand{\zoap}{\mathsf{Z}^{(\omega;p)}_{(a)}}
\newcommand{\wqap}{\mathsf{W}^{(q;p)}_{(a)}}
\newcommand{\zqap}{\mathsf{Z}^{(q;p)}_{(a)}}
\newcommand{\wnqp}{\mathsf{W}^{(q;p_{1},\cdots,p_{n})}_{(a_{1},\cdots,a_{n})}}
\newcommand{\znqp}{\mathsf{Z}^{(q;p_{1},\cdots,p_{n})}_{(a_{1},\cdots,a_{n})}}

\newcommand{\ba}{\boldsymbol{\alpha}}
\newcommand{\bb}{\boldsymbol{\beta}}
\newcommand{\bs}{\boldsymbol{\Sigma}}
\newcommand{\bl}{\boldsymbol{\Lambda}}
\newcommand{\br}{\boldsymbol{\gamma}}

\numberwithin{equation}{section}

\begin{document}

\title{Local times for spectrally negative L\'evy processes\footnote{BL's research is supported by
NNSF (11601243) and the Fundamental Research Funds for the Central Universities.
BL's and XZ's research are supported by NSERC (RGPIN-2016-06704). XZ's research are supported by NNSF (11731012).}}

\author{Bo Li 
\footnote{School of Mathematics and LPMC, Nankai University. 
\href{mailto:libo@nankai.edu.cn}{libo@nankai.edu.cn}}
\and Xiaowen Zhou \footnote{Department of Mathematics and Statistics, Concordia
University.
\href{mailto:xiaowen.zhou@concordia.ca}{xiaowen.zhou@concordia.ca}}}

\maketitle
\begin{abstract}
For
spectrally negative L\'evy processes, adapting an approach from \cite{BoLi:sub1} we identify joint
 Laplace transforms involving local times evaluated at either the first passage times, or independent exponential times, or inverse local times. The Laplace transforms are expressed in terms of the associated scale functions.
 Connections are made  with the permanental process and the Markovian loop soup measure.
\end{abstract}



\section{Introduction}


Occupation times and local times have been well studied for diffusions. But for spectrally negative L\'evy processes (SNLPs in short) the systematic study of occupation times only started 
a few years ago.
 During the last several years there have been
a few papers on Laplace transforms of occupation times for SNLPs, which stem from their applications in risk theory and finance and are also of theoretical interest; see for example \cite{Gerber2012:omega, Landriault2011:occupationtime:levy, Loeffen2014:occupationtime:levy,Zhou2014:occupationtime:levy,Zhou2015:occupationtime:levy,BoLi:sub1}. Among them using a perturbation approach
\cite{Landriault2011:occupationtime:levy} studied the occupation times of semi-infinite intervals for spectrally negative L\'{e}vy processes,
for the occupation times spent in a finite interval, using a strong approximation approach
\cite{Loeffen2014:occupationtime:levy} identified Laplace transforms until first passage times, and
\cite{Zhou2014:occupationtime:levy,Zhou2015:occupationtime:levy} investigated the joint Laplace transforms on occupation times with a different Poisson approach.
The associated resolvent measure was found in
\cite{Guerin2014:occupationmeasure:levy}.

Given the previous results in \cite{Loeffen2014:occupationtime:levy} on
 Laplace transforms of occupation times
spent by a SNLP over finite intervals,
for example, for $0<a<b<c$ and $x\in(0,c)$
\begin{equation}\label{ratio}
\Em_{x}\Big(\exp\big(-p \tau_{c}^{+}- q \int_{0}^{\tau_{c}^{+}}
\mathbf{1}_{\{X_{t}\in(a,b)\}}\,dt\big); \tau_{c}^{+}<\tau_{0}^{-}\Big)=
\frac{W^{(p,q+p,p)}_{(a,b)}(x)}{W^{(p,q+p,p)}_{(a.b)}(c)},
\end{equation}
where $W^{(p,q+p,p)}_{(a,b)}$ is an auxiliary function to be introduced in Section \ref{sec:5.1}
and expressed in terms of the classical scale functions for the SNLP.
The next natural question is to find Laplace transforms of occupation densities, or local times, which can be obtained by taking appropriate limits on
 the associated occupation times.
But  it is not clear to us how to identify the limit of the ratio in \eqref{ratio}  which involves the  asymtotic behavior of $W^{(p,p+\frac{q}{2\varepsilon},p)}_{(a-\varepsilon,a+\varepsilon)}(x)$ as $\varepsilon\to0+$.

On the other hand, by generalizing the Poisson approach, \cite{BoLi:sub1} recently further obtained Laplace transforms of weighted occupation times for SNLPs, which are expressed in terms of the unique solutions to integral equations that involve the scale functions and the weight functions. The integral equations allow to rigorously identify limits of the solutions as the weight function converges to a delta function, which gets around the above-mentioned difficulty on the scale functions and produces the Laplace transforms of the occupation densities.

Applying the results in \cite{BoLi:sub1},  for a SNLP we implement the above mentioned alternative  approach to find joint Laplace transforms on the local time process either at some stopping times or at independent exponential times for a SNLP, which we summarize in the following.

Let $X=(X_{t})_{t\geq 0}$ be a one-dimensional spectrally negative
L\'{e}vy process, i.e. a L\'evy process with no positive
jumps. We are first interested in the joint Laplace
transforms of
\[
\big(\tau_{b}^{+}, l(a,\tau_{b}^{+})\big),
\quad
\big(\tau_{c}^{-}, l(a,\tau_{c}^{-})\big)
\quad\text{and}\quad
\big(X(e_{q}), l(a,e_{q})\big),
\]
where $e_{q}$ is an exponential random variable with parameter $q$
independent of $X$, $l(a,\cdot)$ is the local time of $X$ at level
$a$, and
\[
\tau_{x}^{+}:=\inf\{t\geq 0, X_{t}>x\}
\quad\text{and}\quad
\tau_{x}^{-}:=\inf\{t\geq 0, X_{t}<x\}
\]
 with the convention $\inf\emptyset=\infty $ are the first passage times of level $x$.
Joint Laplace transforms for local times at different levels are also obtained. All our results are expressed in terms of scale functions for {the} process $X$.

The local times for Markov processes can also be studied via permanental processes; see \cite{EISENBAUM20091401,fitzsimmons2014} and references therein.
As applications of our results, we can reprove
a known connection between the local time process and {the} permanental process for {SNLPs}. An expression in terms of the scale function is also found for the joint Laplace transform of the local time process under the loop soup measure.



The rest of the paper is arranged as follows.
In Section \ref{sec:2}, we quickly review the scale functions and some well known fluctuation identities of SNLPs.
Our main results on Laplace transforms of local times are presented in Section \ref{sec:3}
where we consider SNLPs with sample paths of unbound variation. The Laplace transforms of joint occupation times together with their connections with the permanental processes and the loop soup measure are further studied in this section. Section \ref{sec:4} contains several examples.
 Proofs  of the main results together with several preliminary results on the $\omega$-scale functions and the associated integral equations  are deferred to Section \ref{sec:proof}.

\section{Preliminaries}\label{sec:2}
We firstly briefly review the theory of spectrally negative L\'{e}vy processes, the associated scale functions, some fluctuation identities and the local times.
For further details, we refer the readers to  \cite{Bertoin96:book} and \cite{Kyprianou2014:book:levy}.

Let $X=(X_{t})_{t\geq 0}$ be a spectrally negative L\'{e}vy process, i.e.
 a one-dimensional stochastic process with stationary and independent increments and with no positive jumps. We exclude the case that $X$ is the negative of a subordinator. Its Laplace transform exists and is specified by
\[
\Em\big(\exp(\theta X_{t})\big)=\exp(\psi(\theta) t), \quad\text{for} \,\, \theta\geq 0.
\]
The function $\psi(\theta)$, known as the Laplace exponent of $X$, is continuous and strictly convex on $\mathbb{R}^{+}$ and given by the L\'{e}vy-Khintchine formula:
\[
 \psi(\theta)=\frac{\sigma^{2}}{2}\,\theta^{2}+ \gamma \theta
 +\int_{(-\infty, 0)}(e^{\theta x}-1-\theta
 x\, \boldsymbol{1}_{\{x\geq -1\}})\,\Pi (d x),
\]
where $\gamma\in\mathbb{R}$, $\sigma\geq 0$ and the L\'{e}vy measure $\Pi$ is a $\sigma$-finite measure on $(-\infty,0)$ such that $\int_{(-\infty, 0)} (1\wedge x^{2})\,\Pi (dx)<\infty$.

For $q\geq 0$, the $q$-scale function $W^{(q)}$
is a  continuous and increasing function on $[0,\infty)$, $W^{(q)}(x)=0$ for $x<0$
 and
\begin{equation}
\int_{0}^{\infty} e^{-sy} W^{(q)}(y)\,dy=\frac{1}{\psi(s)-q}, \quad\text{for $s>\Phi(q)$}
\end{equation}
where $\Phi(q):=\sup\{s\geq 0, \psi(s)=q\}$ denotes the right inverse of $\psi$. With the scale function $W^{(q)}$ defined, we can define another scale function by
\begin{equation}\label{Z-W}
Z^{(q)}(x):=1+ q \int_{0}^{x}W^{(q)}(y)\,dy, \quad\text{for $x\in\mathbb{R}$}.
\end{equation}
We write $W(x)=W^{(0)}(x)$ and $Z(x)=Z^{(0)}(x)$ whenever $q=0$.
It is known that as $x\to\infty$,
\begin{equation}\label{eqn:limit:wz}
\frac{W^{(q)}(x-a)}{W^{(q)}(x)}\to e^{-\Phi(q)a},
\quad
e^{-\Phi(q)x}W^{(q)}(x)\to\Phi'(q)
\quad \text{and}\quad
\frac{Z^{(q)}(x)}{W^{(q)}(x)}\to \frac{q}{\Phi(q)},
\end{equation}
 where for $q=0$ and $\Phi(0)=0$, we understand that ${q}/{\Phi(q)}={1}/{\Phi'(0)}=\psi'(0)\in[0,\infty)$ and further, $\Phi'(0)=\infty$ for $\psi'(0)=0$.
We refer to \cite{Kuznetsov2012:scalefunction, Hubalek2011:scalefunction:examples} for a more detailed discussions and examples of scale functions.

For simplicity, our main results focus on the case of $W(0)=0$ in which the SNLP $X$ has sample paths of unbounded variation,
and which is equivalent to $\sigma>0$ or $\D \int_{-1}^{0}|x| \Pi(dx)=\infty$.

Throughout the paper, the probability law of $X$ for $X_{0}=x$ is denoted by $\Pm_{x}$ and the corresponding expectation by $\Em_{x}$.
Write $\Pm$ and $\Em$ whenever $x=0$.
For $x\in[0,b]$, the solutions to the two-sided exit problems for $X$ are given by
\begin{equation}\label{one-sided}
\Em_{x}\big(e^{-q \tau_{b}^{+}}; \tau_{b}^{+}<\tau_{0}^{-}\big)
= \frac{W^{(q)}(x)}{W^{(q)}(b)}
\end{equation}
and
\begin{align}
\Em_{x}\big(e^{-q \tau_{0}^{-}}; \tau_{0}^{-}<\tau_{b}^{+}\big)
=&\ Z^{(q)}(x)- \frac{W^{(q)}(x)}{W^{(q)}(b)}Z^{(q)}(b).
\end{align}
For $q\geq0$, the $q$-resolvent of $X$ killed when first exiting {the} interval $[0,b]$ is expressed as
\begin{align}
U^{(q)}(x,dy):=&\ \int_{0}^{\infty} e^{-q t}
\Pm_{x}\big(X_{t}\in\,dy, t<\tau_{b}^{+}\wedge\tau_{0}^{-}\big)\,dt\non\\
=&\ \Big(\frac{W^{(q)}(x)}{W^{(q)}(b)}W^{(q)}(b-y)-W^{(q)}(x-y)\Big)\,dy,
\quad \text{$x,y\in(0,b)$}.
\end{align}

Given a SNLP $X$, its occupation measure is absolutely continuous with respect to Lebesgue measure {$\Pm$-a.s.} with its occupation density, or local time, $l(x,t)$ defined as
\begin{equation}
\label{defn:localtime}
l(x,t):=\limsup_{\varepsilon\to0+}\frac{1}{2\varepsilon} \int_{0}^{t} \boldsymbol{1}_{\{|X_{s}-x|\leq \varepsilon\}}\,ds
\end{equation}
for $x\in\mathbb{R}, t\geq 0$.
If {$W(0)=0$}, the origin is regular for $X$,
the convergence of \eqref{defn:localtime} is in $L^{2}(\Pm)$, which  holds uniformly on any compact interval of time;
see \cite[Chapter V]{Bertoin96:book}.
For $x\in\mathbb{R}$,
$\big(l(x,t)\big)_{t\geq 0}$ defines a
continuous and
increasing process which plays an essential role in defining the process of
excursions of the sample path away from $x$.
Under further conditions, for example \textit{the majorizing measure condition} in \cite[Chapter V]{Bertoin96:book}, there is a version of random field $\big(l(x,t)\big)_{x\in\mathbb{R},t\geq 0}$ that is jointly continuous in $(x,t)$;
see \cite{Bertoin96:book, Pardo2015:exursion} for more details.

 If {$W(0)>0$}, then the origin is irregular for $X$ which has sample paths of bounded variation, and
the local time 
can be defined  as
\[
d\times l(x,t)=\frac{1}{2} \big(\boldsymbol{1}_{\{X_{0}=x\}}+ \boldsymbol{1}_{\{X_{t}=x\}}\big)
+ \#\{s\in(0,t): X_{s}=x\}, \quad  t>0,
\]
where $d= W(0)^{-1}> 0$ is the drift coefficient of process $X$.
In this case, the local time can be evaluated more directly by simply counting the number of hitting times, and we leave it to the interested readers.

\section{Main results}\label{sec:3}
Given the notation in the previous section, we are ready to present our main results.
The proofs of Lemma \ref{lem:hitting} and the Theorems are deferred to Section \ref{sec:proof}.
For the rest of this section, $p,q\geq 0$ and $a,b,c\in\mathbb{R}$ are constants satisfying $c<a<b$.


\subsection{Local time at first passage times}
We first consider  the local times at the fixed point $a$.
Define auxiliary {generalized} scale functions as follows,
\begin{equation}\label{eqn:1}
\left\{
\begin{split}
\wqap(x,y):=&\ W^{(q)}(x-y)+ p W^{(q)}(x-a) W^{(q)}(a-y),\\
\zqap(x,c):=&\ Z^{(q)}(x-c)+ p W^{(q)}(x-a) Z^{(q)}(a-c),
\end{split}\right.
\end{equation}
for $x,y\in\mathbb{R}$.
More general $\omega$-scale functions are  to be introduced later in Section \ref{sec:5.1}. Then we can express the following joint Laplace transforms involving the local time using the {generalized} scale functions.
\begin{thm}
[Fluctuation identities involving local times]
\label{thm:loc:q}
For any $x\in[c,b]$, we have
\begin{align*}
 \Em_{x}\big(e^{-q\tau_{b}^{+}-p l(a,\tau_{b}^{+})}; \tau_{b}^{+}<\tau_{c}^{-}\big)
=&\ \frac{\wqap(x,c)}{\wqap(b,c)}
\end{align*}
and
\begin{align*}
 \Em_{x}\big(e^{-q\tau_{c}^{-}-p l(a,\tau_{c}^{-})}; \tau_{c}^{-}<\tau_{b}^{+}\big)
=&\ \zqap(x,c)- \frac{\wqap(x,c)}{\wqap(b,c)} \zqap(b,c).
\end{align*}
We also have the following resolvent density. For any $y\in(c,b)$,
\begin{align*}
&\ \int_{0}^{\infty} e^{-q t}
\Em_{x}\big(e^{-pl(a,t)}; X(t)\in dy, t<\tau_{b}^{+}\wedge\tau_{c}^{-}\big)\,dt\\
=&\ \Big(\frac{\wqap(x,c)}{\wqap(b,c)} \wqap(b,y)- \wqap(x,y)\Big)\,dy.
\end{align*}
\end{thm}



Taking $q=0$ in Theorem \ref{thm:loc:q}, we have for $x\in(c,b)$
\begin{align*}
& \Em_{x}\big(e^{-p l(a,\tau_{b}^{+})}\big|
\tau_{b}^{+}<\tau_{c}^{-}\big)
= \frac{W(b-c)}{W(x-c)}\times\frac{W(x-c)+ p W(x-a)W(a-c)}{W(b-c)+ pW(b-a)W(a-c)}\\
=&\ \frac{ W(x-a) W(b-c)}{W(x-c)W(b-a)}+
\Big(1- \frac{ W(x-a)W(b-c)}{W(x-c)W(b-a)}\Big)
\frac{W(b-c)}{W(b-c)+ pW(b-a)W(a-c)}.
\end{align*}
 Inverting the above Laplace transform gives the following distribution of $l(a,\tau_{b}^{+})$.
\begin{cor}\label{cor:7} Given $x\in(c,b)$, we have for $t>0$,
\begin{equation}\label{cor:7a}
\Pm_{x}\big(l(a,\tau_{b}^{+})=0\big|\tau_{b}^{+}<\tau_{c}^{-}\big)=
\frac{ W(x-a) W(b-c)}{W(x-c)W(b-a)}
\end{equation}
and
\begin{equation}\label{cor:7b}
\Pm_{x}\big(l(a,\tau_{b}^{+})>t \big| l(a,\tau_{b}^{+})>0, \tau_{b}^{+}<\tau_{c}^{-}\big)
= \exp\Big(\frac{-W(b-c)t}{W(b-a)W(a-c)}\Big).
\end{equation}
\end{cor}


For $a\in\mathbb{R}$ put
\begin{equation}
\tau^{\{a\}}:=\inf\{t>0, X_{t}=a\}.
\end{equation}
Since $l(a,t)=0$ for $t<\tau^{\{a\}}$,
the left hand side of \eqref{cor:7a} in Corollary \ref{cor:7} equals to
$\Pm_{x}(\tau_{b}^{+}<\tau^{\{a\}}|\tau_{b}^{+}<\tau_{c}^{-})$.
Note that
$l(a,t)=l(a,t-\tau^{\{a\}})\circ\theta_{\tau^{\{a\}}}$ on the event  $\{t>\tau^{\{a\}}\}$,
the valuations at local time $l(a,\cdot)$ under $\Pm_{x}$ can be converted to those under $\Pm_{a}$ by applying the strong Markov property at $\tau^{\{a\}}$,
and where we need the following lemma concerning the hitting time.

\begin{lem}\label{lem:hitting} For any $x, a\in(c,b)$, we have
\begin{equation}
\Em_{x}\big(e^{- q \tau^{\{a\}}}; \tau^{\{a\}}<\tau_{b}^{+}\wedge\tau_{c}^{-}\big)
=\frac{W^{(q)}(x-c)}{W^{(q)}(a-c)}- \frac{W^{(q)}(x-a)W^{(q)}(b-c)}{W^{(q)}(b-a)W^{(q)}(a-c)}.
\end{equation}
\end{lem}

Denoting by $e_{q}$ an exponential variable with parameter $q$ and independent of  $X$, we have
\begin{cor}\label{cor:1}
Given $b>a>c$, we have for $p,q> 0$
\[
\Em_{a}\big(e^{-p l(a,e_{q}\wedge\tau_{b}^{+}\wedge\tau_{c}^{-})}\big)
= \frac{W^{(q)}(b-c)}{W^{(q)}(b-c)+ p W^{(q)}(b-a) W^{(q)}(a-c)},
\]
i.e. $l(a,e_{q}\wedge\tau_{b}^{+}\wedge\tau_{c}^{-})$ has
 an exponential distribution.
\end{cor}
\begin{proof}
It follows from
from Theorem \ref{thm:loc:q} that
\begin{align}
&\ \Em_{a}\big(e^{-pl(a,\tau_{b}^{+})}; \tau_{b}^{+}<\tau_{c}^{-}\wedge e_{q}\big)
=\Em_{a}\big(e^{-q\tau_{b}^{+}-p l(a,\tau_{b}^{+})}; \tau_{b}^{+}<\tau_{c}^{-}\big)\non\\
=&\ \frac{W^{(q)}(a-c)}
{W^{(q)}(b-c)+ p W^{(q)}(b-a) W^{(q)}(a-c)},\label{UB:b+}\\
&\ \Em_{a}\big(e^{-pl(a,\tau_{c}^{-})}; \tau_{c}^{-}<\tau_{b}^{+}\wedge e_{q}\big)
= \Em_{a}\big(e^{-q\tau_{c}^{-}-p l(a,\tau_{c}^{-})}; \tau_{c}^{-}<\tau_{b}^{+}\big)\non\\
=&\ \frac{Z^{(q)}(a-c)W^{(q)}(b-c)- Z^{(q)}(b-c)W^{(q)}(a-c)}
{W^{(q)}(b-c)+ p W^{(q)}(b-a) W^{(q)}(a-c)},\label{UB:c-}
\end{align}
and
\begin{align}
&\ \Em_{a}\big(e^{-p l(a,e_{q})}; X_{e_{q}}\in dy, e_{q}<\tau_{b}^{+}\wedge\tau_{c}^{-}\big)\non\\
=&\ q\Big(\frac{W^{(q)}(b-y)W^{(q)}(a-c)- W^{(q)}(a-y)W^{(q)}(b-c)}
{W^{(q)}(b-c)+ p W^{(q)}(b-a) W^{(q)}(a-c)}\Big)\,dy.\label{UB:joint}
\end{align}
Integrating in $y$ over $[c,b]$ on both sides of \eqref{UB:joint} gives
\begin{align}
&\ \Em_{a}\big(e^{-p l(a,e_{q})}; e_{q}<\tau_{b}^{+}\wedge\tau_{c}^{-}\big)\non\\
=&\ \frac{W^{(q)}(a-c)(Z^{(q)}(b-c)-1)- W^{(q)}(b-c)(Z^{(q)}(a-c)-1)}
{W^{(q)}(b-c)+ p W^{(q)}(b-a) W^{(q)}(a-c)}.\label{UB:e}
\end{align}
Adding up \eqref{UB:b+}, \eqref{UB:c-} and \eqref{UB:e} yields the result.
\end{proof}
Note that \eqref{UB:b+} can also be obtained from \eqref{cor:7b} by a standard change of measure argument.
We also have the following results on joint distribution of $X_{e_{q}}$ and $l(a,e_{q})$, where
the first identity coincides with  the expression of resolvent density for process
$X$ killed at the two-sided exit time,  as one would expect.

\begin{cor}\label{cor:3} Conditioning on the event $\{e_{q}<\tau_{b}^{+}\wedge\tau_{c}^{-}\}$, $X_{e_{q}}$ and $l(a,e_{q})$ are independent and follow the respective distributions
\[
q^{-1}\Pm_{a}\big(X_{e_{q}}\in dy, e_{q}<\tau_{b}^{+}\wedge\tau_{c}^{-}\big)=
\Big(\frac{W^{(q)}(a-c)}{W^{(q)}(b-c)}W^{(q)}(b-y)- W^{(q)}(a-y)\Big)\,dy
\]
and
\[
\Pm_{a}\big(l(a,e_{q})\in dt \big| e_{q}<\tau_{b}^{+}\wedge\tau_{c}^{-}\big)=
\frac{W^{(q)}(b-c)}{W^{(q)}(b-a)W^{(q)}(a-c)} e^{-\frac{W^{(q)}(b-c) t}{W^{(q)}(b-a)W^{(q)}(a-c)}}\,dt.
\]
\end{cor}
\begin{proof}
Inverting the Laplace transform in \eqref{UB:joint} we have
\begin{align*}
&\ \Pm_{a}\big(l(a,e_{q})\in dt, X_{e_{q}}\in dy, e_{q}<\tau_{b}^{+}\wedge\tau_{c}^{-}\big)\\
=&\ q\Big(\frac{W^{(q)}(a-c)}{W^{(q)}(b-c)}W^{(q)}(b-y)- W^{(q)}(a-y)\Big)\,dy\\
& \quad\times \Big(\frac{W^{(q)}(b-c)}{W^{(q)}(b-a)W^{(q)}(a-c)} e^{-\frac{W^{(q)}(b-c) t}{W^{(q)}(b-a)W^{(q)}(a-c)}}\Big)\,dt.
\end{align*}
Conditioning on the event $\{e_{q}<\tau_{b}^{+}\wedge\tau_{c}^{-}\}$ finishes the proof.
\end{proof}

Letting $c$ and $b$ tend to infinity, respectively, in \eqref{UB:b+}, \eqref{UB:c-} and \eqref{UB:joint},
from \eqref{eqn:limit:wz} we have the next result.
\begin{cor}\label{cor:2}
Given $b>a>c$, we have for $p,q> 0$,
\begin{align*}
\Em_{a}\big(e^{-q\tau_{b}^{+}-p l(a,\tau_{b}^{+})}; \tau_{b}^{+}<\infty\big)
=&\ \frac{1}{e^{\Phi(q)(b-a)}+ p W^{(q)}(b-a)},\\
\Em_{a}\big(e^{-q \tau_{c}^{-}-p l(a,\tau_{c}^{-})}; \tau_{c}^{-}<\infty\big)
=&\ \frac{Z^{(q)}(a-c) - \frac{q}{\Phi(q)} W^{(q)}(a-c)}{1+ p e^{\Phi(q)(c-a)}W^{(q)}(a-c)}
\end{align*}
and
\begin{align*}
\Em_{a}\big(e^{-p l(a,e_{q})}\big)
=&\ \frac{1}{1+p\Phi'(q)}.
\end{align*}
\end{cor}


%
\paragraph{Joint local times}
 Besides the local times at one level, we are also interested in
 the local times at different levels. To which end,
 we need two more auxiliary functions,
 which also turn out to be useful in studying the local times at inverse local time.

 For the rest of this section,
 let $\{p_{k}>0\}_{k\geq1}$ be a positive sequence
 and $\{a_{k}\}_{k\geq1}$ be an increasing sequence with $a_{1}>c$.
 We introduce the following auxiliary functions,
 which {are} defined inductively and {generalize} \eqref{eqn:1},
 for $x,y\in\mathbb{R}$ and $k\geq2$ ,
 \begin{equation}\label{eqn:2}
 \left\{ \begin{split}
 \mathsf{W}^{(q;p_{1},\cdots,p_{k})}_{(a_{1},\cdots,a_{k})}(x,y):=&\
\mathsf{W}^{(q;p_{1},\cdots,p_{k-1})}_{(a_{1},\cdots,a_{k-1})}(x,y)+ p_{k} W^{(q)}(x-a_{k})\mathsf{W}^{(q;p_{1},\cdots,p_{k-1})}_{(a_{1},\cdots,a_{k-1})}(a_{k},y),\\
\mathsf{Z}^{(q;p_{1},\cdots,p_{k})}_{(a_{1},\cdots,a_{k})}(x,c):=&\
\mathsf{Z}^{(q;p_{1},\cdots,p_{k-1})}_{(a_{1},\cdots,a_{k-1})}(x,c)+ p_{k} W^{(q)}(x-a_{k})\mathsf{Z}^{(q;p_{1},\cdots,p_{k-1})}_{(a_{1},\cdots,a_{k-1})}(a_{k},c).
 \end{split}\right.\end{equation}
Then we have the following result concerning the local times at different levels.

\begin{thm}\label{cor:5}
Let $\wnqp$ and $\znqp$ {be defined as} in \eqref{eqn:2}. We have for $x\in[c,b]$,
\begin{equation}\label{joint_a}
\Em_{x}\Big(\exp\big(-q \tau_{b}^{+}-\sum_{j=1}^{n}p_{j} l(a_{j},\tau_{b}^{+})\big);
\tau_{b}^{+}<\tau_{c}^{-}\Big)
=\frac{\wnqp(x,c)}{\wnqp(b,c)}
\end{equation}
and
\begin{align}\label{joint_b}
&\ \Em_{x}\Big(\exp\big(-q \tau_{c}^{-}-\sum_{j=1}^{n}p_{j} l(a_{j},\tau_{c}^{-})\big);
\tau_{c}^{-}<\tau_{b}^{+}\Big)\non\\
=&\ \znqp(x,c)-\frac{\wnqp(x,c)}{\wnqp(b,c)}
\znqp(b,c).
\end{align}
In addition, we also find  an expression of the resolvent density for $y\in (c, b)$,
\begin{align*}
&\ \int_{0}^{\infty} e^{-q t}
\Em_{x}\Big(\exp\big(-\sum_{j=1}^{n}p_{j} l(a_{j},t)\big);
X(t)\in dy, t<\tau_{b}^{+}\wedge\tau_{c}^{-}\Big)\,dt\\
=&\ \Big(\frac{\wnqp(x,c)}{\wnqp(b,c)} \wnqp(b,y)- \wnqp(x,y)\Big)\,dy.
\end{align*}
\end{thm}

%
\subsection{Local times at inverse local time}
Define the right continuous inverse local time at level $a$ by
\begin{equation}
l^{-1}(a,t):=\inf\{s\geq 0, l(a,s)>t\}, \quad t\geq 0.
\end{equation}
Notice that $\{l(a,T)>t\}=\{l^{-1}(a,t)<T\}$ for every $t>0$ and stopping time $T$.
One can check the following results directly,
by inverting the Laplace transform of $l(a,\cdot)$
from \eqref{UB:b+}, \eqref{UB:c-}, \eqref{UB:e} and Corollary \ref{cor:1},
respectively.
\begin{cor}\label{cor:4} For any $b>a>c$ and $t> 0$,
\begin{align*}
&\ \Pm_{a}\big(l^{-1}(a,t)<\tau_{b}^{+}\big| \tau_{b}^{+}<e_{q}\wedge\tau_{c}^{-}\big)
=\Pm_{a}\big(l^{-1}(a,t)<e_{q}\big| e_{q}<\tau_{b}^{+}\wedge\tau_{c}^{-}\big)
\\=&\
\Pm_{a}\big(l^{-1}(a,t)<\tau_{c}^{-}\big| \tau_{c}^{-}<e_{q}\wedge\tau_{b}^{+}\big)
=
\Pm_{a}\big(l^{-1}(a,t)<e_{q}\wedge\tau_{b}^{+}\wedge\tau_{c}^{-}\big)\\
=&\ \exp\big(-\frac{W^{(q)}(b-c) t}{W^{(q)}(b-a)W^{(q)}(a-c)}\big).
\end{align*}
\end{cor}

%
%

Notice that the above result on $l(a,\tau_{b}^{+})$ coincides with Corollary \ref{cor:7}. It is known that $\{l^{-1}(a,t),t\geq 0\}$ under $\Pm_{a}(\cdot)$ is a subordinator with Laplace exponent $(u^{(\lambda)}(0))^{-1}$, where $u^{(\lambda)}(y)$ is the $\lambda$-resolvent density of $X$ at $y$; see \cite[Chapter V.1.4]{Bertoin96:book}. Since
\[
u^{(\lambda)}(y)=\Big(\frac{1}{dx}\int_{0}^{\infty}e^{-\lambda t} \Pm(X_{t}\in dx)\,dt\Big)\Big|_{x=y}=\Phi'(\lambda)e^{-\Phi(\lambda)y}-W(-y),
\]
we have $(u^{(\lambda)}(0))^{-1}=\Phi'(\lambda)^{-1}$ which agrees with Corollary \ref{cor:2}.
Moreover,
we have the following joint Laplace transform of local times evaluated at the inverse local time, which extends Corollary \ref{cor:4}.

\begin{thm}\label{cor:6}
Let $a_{n}<b$ for some $n\in\mathbb{N}$
and $\wnqp$ {be defined as} in \eqref{eqn:2}.
Then
\begin{align}
&\ \Em_{a}\Big(\exp\Big(-q l^{-1}(a,t)-\sum_{j=1}^{n}p_{j} l(a_{j},l^{-1}(a,t))\Big); l^{-1}(a,t)<\tau_{b}^{+}\wedge\tau_{c}^{-}\Big)\non\\
=&\ \exp\Big(\frac{-\wnqp(b,c)t}{\wnqp(b,a) \wnqp(a,c)}\Big).
\end{align}
\end{thm}

{ The local time process $l(a,\cdot)$ is the time scale of
the  excursion process that is a Poisson point process  taking values from the space of excursion sample paths of $X$ away from $a$.}
{Excursion theory finds successful applications in the study of
local times. Many results in this section can also be proved or interpreted using the excursion theory.
For example,
{the exponential result in} \eqref{cor:7b} follows from excursion theory since
$l(a,\tau_{b}^{+})$ is the time of the first excursion started at level $a$ that first leaves the interval $(c, b)$ from above.
{So does the exponential result in Theorem \ref{cor:6}.}
{The independence between $X(e_{q})$ and $l(a,e_{q})$}
 in Corollary \ref{cor:3} were proved in \cite[Proposition 3.2]{Mansuy2008:book} for
Brownian motion, and it is {actually} a consequence of a well known property
of the process of excursions of $X$ away from $a$ as a  Poisson point process indexed by the local time of $X$ at level $a$,
as well as the independence shown in Corollary \ref{cor:4}.}
 The associated results on excursion measures can also be derived from the corollaries using the excursion theory.  We leave the details to interested readers.

\subsection{Matrix expressions
and
permanental processes}

For fixed $n\in\mathbb{N}$, let $\wnqp$ and $\znqp$ be as defined in \eqref{eqn:2}.
Here we first present matrix expressions for the auxiliary functions,
which  facilitates the later computations.
By further looking into the joint local times,
we connect our results with the theory of Markovian loop soups and permanental processes,
see for example \cite{EISENBAUM20091401,fitzsimmons2014}.

In this subsection,
we denote by
$\boldsymbol{\nu}=(\nu_{j})_{1\leq j\leq n}$ an $n$-dimensional vector,
by $\mathbf{M}=(m_{ij})_{1\leq i, j\leq n}$ an $n\times n$-matrix with respective entries,
by $\mathbf{M}^{\mathrm{T}}$ the transpose of $\mathbf{M}$
and by $\mathbf{I}$ the identity matrix.
We also need the following notations:
\begin{equation}\label{eqn:3.3.g}
\begin{split}
&\bs:=\big(W^{(q)}(a_{i}-a_{j})\big)_{1\leq i, j\leq n}, \quad
\ba(x):=\big(W^{(q)}(x-a_{i})\big)_{1\leq i\leq n},\\
&
\bb(y):=\big(W^{(q)}(a_{i}-y)\big)_{1\leq i\leq n},
\quad\br:=\big(Z^{(q)}(a_{i}-c)\big)_{1\leq i\leq n},
\end{split}
\end{equation}
and the diagonal matrix $\bl=\text{diag}(p_{1},p_{2}, \cdots,p_{n})$. Then we have the following representations of the generalized scale functions.

\begin{prop}\label{prop:1}
For $x,y\in\mathbb{R}$, we have
\begin{align}
\wnqp(x,y)=&\ \det\begin{pmatrix}
W^{(q)}(x-y) & \ba^{\mathrm{T}}(x)\\
-\bl \bb(y) & \mathbf{I}-\bl\bs\\
\end{pmatrix}\label{eqn:3}
\end{align}
and
\begin{align}
\znqp(x,c)=&\ \det\begin{pmatrix}
Z^{(q)}(x-c) & \ba^{\mathrm{T}}(x)\\
-\bl \br & \mathbf{I}-\bl\bs
\end{pmatrix}.\label{eqn:4}
\end{align}
\end{prop}
Notice that $W^{(q)}(x-y)=0$ for $x\leq y$,
$\big(\mathbf{I}-\bl\bs\big)$ is a  lower triangular matrix with entries $1$ on the diagonal.
Suppose that $a\in[a_{m},a_{m+1})$ for some $m\leq n$.
It follows from Proposition \ref{prop:1} that
\[\wnqp(a,c)=W^{(q;p_{1},\cdots,p_{m})}_{(a_{1},\cdots,a_{m})}(a,c)\quad
\text{and}\quad
\wnqp(b,a)=W^{(q;p_{m+1},\cdots,p_{n})}_{(a_{m+1},\cdots,a_{n})}(b,a).\]

By the Ray-Knight theorem, see for example \cite{Mansuy2008:book},
 the local times of Brownian motion are identical in law to {a squared} Beseel process.
A similar relationship was shown in \cite{Eisenbaum2000} between symmetric Markov processes
and a family of squares of Gaussian processes.
For  general Markov process, the representation the local times involves  a so called permanental process.
A permanental process with parameter set $\mathcal{E}$ is a positive process whose finite-dimensional
Laplace transform is given by a negative power of a determinant.

\begin{defn}
A positive-valued  process $\{\xi_{x}, x\in\mathcal{E}\}$ is called a permanental process
with kernel $\big(g(x,y), x,y\in\mathcal{E}\big)$ and index $\beta>0$,
if its finite-dimensional Laplace transforms satisfy,
for every {$m\in\mathbb{N}$}, $(p_{1}, \cdots, p_{m})\in\mathbb{R}^{+,m}$ and $(x_{1},\cdots,x_{m})\in\mathcal{E}^{m}$,
\[
\Em\Big(\exp\big(-\frac{1}{2}\sum_{j=1}^{m} p_{j}\xi_{x_{j}}\big)\Big)
= \big(\det(\mathbf{I}+ \bl \mathbf{G})\big)^{-1/\beta}
\]
where $\bl=\text{diag}(p_{1},\cdots,p_{m})$ is a diagonal matrix
and $\mathbf{G}=\big(g(x_{i},x_{j})\big)_{1\leq i,j\leq m}$.
\end{defn}

An equivalent definition for the permanental process can be found in \cite{fitzsimmons2014}
by specifying its moments, which is called an $\alpha$-permanental process with $\alpha\beta=1$.
A necessary and sufficient condition on matrices $\mathbf{G}$ and $\beta>0$
for the existence of the corresponding permanental process was established in
\cite{articledavid}.


As an application of the results in this paper, we first present a new proof of
the following known identity on the relationship between the local time process and the permanental process.
Let $Y$ be a transient Markov processes on $\mathcal{E}=\mathbb{R}$,
$g(x,y)$ be the potential density with respect to a reference measure $m(dx)$,
and $\{l(x,t),x\in\mathcal{E}, t\geq0\}$ be a normalized local time process,
that is $g(x,y)=\Em_{x}(l(y,\infty))$.
 It was shown in \cite{fitzsimmons2014}  that there exists a permanental process
$(\xi_{x},x\in\mathcal{E})$ associated with $Y$,
with kernel $g(x,y)$ and index $\beta=2$, and is independent of $Y$.
In particular, it was shown in \cite[Thm 3.2]{EISENBAUM20091401} that
for every $a\in\mathcal{E}$ with $g(a,a)>0$,
\begin{equation}\label{eqn:3.3.2}
\big(l(x,\infty)\big|_{\tilde{\Pm}_{a}}+ \frac{1}{2}\xi_{x}, x\in\mathcal{E}\big)
\stackrel{d}{=}
\big(\frac{1}{2}\xi_{x}, x\in\mathcal{E}\big)\big|_{Q},
\end{equation}
where $Q$ is a change of measure defined by $Q(\eta):=\frac{\Em(\xi_{a} \eta)}{g(a,a)}$
for any random variable $\eta$,
and where $\tilde{\Pm}_{a}$  is another change of probability measure defined by
\[
d\tilde{\Pm}_{a}\Big|_{\mathscr{F}_{t}}= \frac{g(Y_{t},a)}{g(a,a)}
d\Pm_{a}\Big|_{\mathscr{F}_{t}},
\quad\text{{for all $t>0$}}
\]
which is the {probability law} that $Y$ starts at $a$ and gets killed on its last exit from $a$.
The expectation with respect to $\tilde{\Pm}_{a}$ is denoted by $\tilde{\Em}_{a}$.

\

For the spectrally negative L\'evy process $X$,
we can define an associated transient Markov process
\begin{equation}\label{eqn:3.3.a}
Y:= X\circ k_{\tau_{b}^{+}\wedge\tau_{c}^{-}\wedge e_{q}},
\end{equation}
where $k_{t}$ denotes the killing operator
and $e_{q}$ is, as previously defined, the exponential random variable  independent of $X$.
{Then process $X$ is killed either at its first time of leaving interval $[c,b]$ or at the random time $e_{q}$.}
We are going to reprove \eqref{eqn:3.3.2} using the results obtained in this section.

It is not hard to check that, for any measurable function $f\geq0$,
\begin{align*}
&\ \int_{0}^{\infty}\Em_{x}(f(Y_{t}))\,dt
=\int_{0}^{\infty} \Em_{x}\big(f(X_{t}); t<\tau_{b}^{+}\wedge\tau_{c}^{-}\wedge e_{q}\big)\,dt\\
=&\ \int_{c}^{b} f(y)
\Big(\frac{W^{(q)}(x-c)}{W^{(q)}(b-c)}W^{(q)}(b-y)-W^{(q)}(x-y)\Big)\,dy.
\end{align*}
Then for any $x,y\in(c,b)$, the potential density $g(x,y)$ of $Y$ is given by
 \[g(x,y)=\frac{W^{(q)}(x-c)}{W^{(q)}(b-c)}W^{(q)}(b-y)-W^{(q)}(x-y).\]
And by the definition, the permanental process $\xi$ with index $\beta=2$ satisfies
\begin{equation}\label{eqn:3.3.d}
\Em\big(\exp(-\sum_{j=1}^{n} \frac{p_{j}}{2}\xi(a_{j}))\big)
= \Big(\det\big(\mathbf{I}+ \bl \mathbf{G}\big)\Big)^{-1/2}
\end{equation}
{with $\mathbf{G}=\big(g(a_{i},a_{j})\big)_{1\leq i,j\leq n}$.}
Then {for any $a\in(c,b)$},
\[
\Em\Big(\xi(a)\exp(-\sum_{j=1}^{n} \frac{p_{j}}{2}\xi(a_{j}))\Big)
=
\big(\det \big(\mathbf{I}+\bl \mathbf{G}\big)\big)^{-\frac{3}{2}}
\det\begin{pmatrix}
g(a,a) & \big(g(a,a_{j})\big)^{\mathrm{T}}\\
\bl \big(g(a_{i},a)\big)& \mathbf{I}+ \bl\mathbf{G}
\end{pmatrix}\]
which follows from \eqref{eqn:3.3.d};
also see \cite[eqn(5.8)]{EISENBAUM20091401}.
To show \eqref{eqn:3.3.2} we  only need to show the identity
\begin{equation}\label{eqn:3.3.b}
\tilde{\Em}_{a}\big(\exp\big(
-\sum_{j=1}^{n}p_{j}l(a_{j},\infty)\big)\big)=
\det\begin{pmatrix}
g(a,a) & \big(g(a,a_{j})\big)^{\mathrm{T}}\\
\bl \big(g(a_{i},a)\big)& \mathbf{I}+ \bl\mathbf{G}
\end{pmatrix}
\frac{\big(\det \big(\mathbf{I}+\bl \mathbf{G}\big)\big)^{-1}}{g(a,a)}.
\end{equation}
To this end, we make use of the  relation from \cite{fitzsimmons2014} that
\[
\tilde{\Em}_{a}\big(F\big)
=\frac{1}{g(a,a)} \Em_{a}\Big(\int_{0}^{\infty} F\circ k_{t}\ l(a,dt)\Big)
\]
for any measurable function $F$.
Then the expectation in \eqref{eqn:3.3.b} equals to
\begin{align*}
&\ \frac{1}{g(a,a)}
\Em_{a}\Big(\int_{0}^{\infty} e^{-\sum_{j=1}^{n} p_{j}l(a_{j},t)}
\mathbf{1}_{\{t<\tau_{b}^{+}\wedge\tau_{c}^{-}\wedge e_{q}\}} l(a,dt)\Big)\\
=&\ \frac{1}{g(a,a)}
\int_{0}^{\infty}
\Em_{a}\big(e^{-\sum_{j=1}^{n} p_{j}l(a_{j}, l^{-1}(a,s))};
l^{-1}(a,s)<\tau_{b}^{+}\wedge\tau_{c}^{-}\wedge e_{q}\big)\,ds\\
=&\ \frac{1}{g(a,a)}
\int_{0}^{\infty}\Em_{a}
\big(e^{-q l^{-1}(a,s)- \sum_{j=1}^{n}p_{j}l(a_{j},l^{-1}(a,s))};
l^{-1}(a,s)<\tau_{b}^{+}\wedge \tau_{c}^{-}\big)\,ds
\end{align*}
where a change of variable
and the Fubini theorem are applied for the second equality.
By Theorem \ref{cor:6}, we have
\begin{equation}\label{eqn:3.3.c}
\tilde{\Em}_{a}\big(\exp\big(
-\sum_{j=1}^{n}p_{j}l(a_{j},\infty)\big)\big)=
\frac{\wnqp(b,a)\wnqp(a,c)}{\wnqp(b,c)\times g(a,a)}.
\end{equation}
To show \eqref{eqn:3.3.b} it is thus sufficient to show that
for any $a\in(c,b)$,
\begin{align}
\frac{\wnqp(b,c)}{W^{(q)}(b-c)}=&\
\det\big(\mathbf{I}+\bl \mathbf{G}\big)\label{eqn:3.3.6}
\end{align}
and
\begin{align}
\frac{\wnqp(b,a)\wnqp(a,c)}{W^{(q)}(b-c)}=&\
\det\begin{pmatrix}
g(a,a) & \big(g(a,a_{j})\big)^{\mathrm{T}}\\
\bl \big(g(a_{i},a)\big)& \mathbf{I}+ \bl\mathbf{G}
\end{pmatrix}.\label{eqn:3.3.7}
\end{align}

Recalling the notation in \eqref{eqn:3.3.g}
{we have}
\begin{align*}
g(a,a)=&\
\frac{W^{(q)}(a-c)W^{(q)}(b-a)}{W^{(q)}(b-c)}, &
g(a,a_{j})=&\
\frac{W^{(q)}(a-c)}{W^{(q)}(b-c)}\alpha_{j}(b)
-\alpha_{j}(a),\\
g(a_{i},a)=&\
\frac{W^{(q)}(b-a)}{W^{(q)}(b-c)}\beta_{i}(c)-\beta_{i}(a), &
g(a_{i},a_{j})=&\
\frac{\beta_{i}(c)\alpha_{j}(b)}{W^{(q)}(b-c)}
- \bs_{ij}.
\end{align*}
By Proposition \ref{prop:1}, it gives
\begin{equation*}
\begin{split}
\frac{\wnqp(b,c)}{W^{(q)}(b-c)}
&=\det\begin{pmatrix}
1 & \big(\frac{\alpha_{j}(b)}{W^{(q)}(b-c)}\big)^{\mathrm{T}}\\
-\big(p_{i}\beta_{i}(c)\big) &
\big(\delta_{ij}-p_{i}\bs_{ij}\big)
\end{pmatrix}\\
&= \det\begin{pmatrix}
1 & \big(\frac{\alpha_{j}(b)}{W^{(q)}(b-c)}\big)^{\mathrm{T}}\\
\mathbf{0}&\
\Big(\delta_{ij}+ p_{i}
\big(\frac{\alpha_{j}(b)\beta_{i}(c)}{W^{(q)}(b-c)}-\bs_{ij}\big)\Big)
\end{pmatrix}\\
&= \det\Big(\mathbf{I}
+ \bl \big(\frac{\alpha_{j}(b)\beta_{i}(c)}{W^{(q)}(b-c)}-\bs_{ij}\big)\Big)\\
&=\det\Big(\mathbf{I}+\bl\mathbf{G}\Big),
\end{split}
\end{equation*}
which establishes \eqref{eqn:3.3.6}.

To show \eqref{eqn:3.3.7}, note that
\[
\det\begin{pmatrix}
g(a,a) & \big(g(a,a_{j})\big)^{\mathrm{T}}\\
\bl \big(g(a_{i},a)\big)& \mathbf{I}+ \bl\mathbf{G}
\end{pmatrix}
=- \det\begin{pmatrix}
0&  \big(g(a,a_{j})\big)^{\mathrm{T}} & g(a,a) \\
\mathbf{0}&  \mathbf{I}+ \bl\mathbf{G} & \bl\big(g(a_{i},a)\big)\\
1&
\big(-\frac{\alpha_{j}(b)}{W^{(q)}(b-c)}\big) & -\frac{W^{(q)}(b-a)}{W^{(q)}(b-c)}
\end{pmatrix}.
\]
{For the  matrix on the right hand side of the above equation, first adding to the first row with the $(n+2)$-th row multiplied by $W^{(q)}(a-c)$, and then adding
to every $(i+1)$-th row with the $(n+2)$-th row multiplied by $p_{i}\beta_{i}(c)$, respectively, for $i=1,\ldots,n$,}
we have
\begin{align*}
&\ -\det\begin{pmatrix}
W^{(q)}(a-c) & \big(-\alpha_{j}(a)\big)^{\mathrm{T}}& 0\\
\big(p_{i}\beta_{i}(c)\big)& \big(\delta_{ij}- p_{i}\bs_{ij}\big) & \big(-p_{i}\beta_{i}(a)\big)\\
1 &
\big(-\frac{\alpha_{j}(b)}{W^{(q)}(b-c)}\big)^{\mathrm{T}} & -\frac{W^{(q)}(b-a)}{W^{(q)}(b-c)}
\end{pmatrix}\\
=&\ \frac{1}{W^{(q)}(b-c)}
\det\begin{pmatrix}
W^{(q)}(a-c)& \ba^{\mathrm{T}}(a) & 0\\
-\bl\bb(c)& \mathbf{I}-\bl\bs & -\bl\bb(a)\\
W^{(q)}(b-c)& \ba^{\mathrm{T}}(b) & W^{(q)}(b-a)
\end{pmatrix}.
\end{align*}
{Let $m:=\max\{k\leq n, a_{k}\leq a\}$.} Notice that
$\alpha_{j}(a)=0=\beta_{i}(a)$ for $j>m\geq i$.
The matrix above can be divided into four sub-matrices,
with the upper-right $(m+1)\times (n-m+1)$ sub-matrix equals to $\mathbf{0}$.
We can then prove \eqref{eqn:3.3.7} using \eqref{eqn:3} in Proposition \ref{prop:1} and
the facts that
$\wnqp(a,c)=W^{(q;p_{1},\cdots,p_{m})}_{(a_{1},\cdots,a_{m})}(a,c)$
and
$\wnqp(b,a)=W^{(q;p_{m+1},\cdots,p_{n})}_{(a_{m+1},\cdots,a_{n})}(b,a)$.



The  permanental process (vector) provides a representation
of the law of the local time process for a general Markov process
taking values in space {$\mathcal{E}$}.
Markovian loop soup
provides a  construction of the permanental process from the Markov process with a different approach.
The Markovian loop soup $\mathcal{L}$ defined  in \cite{fitzsimmons2014}
is a Poisson point process associated to the transient Markov process $Y$. Its characteristic measure, denoted by $\mu$ and called the loop measure, on a space   
$\mathscr{E}$ of right-continuous paths $\epsilon$ taking values in $\mathcal{E}\cup \Delta$ with $\Delta$ denoting  a cemetery state,
is defined by, for any function $F$  on $\mathscr{E}$,
\[
\mu(F):=\int_{\mathcal{E}} \Pm_{x}\big(\int_{0}^{\infty} \frac{1}{t} F\circ k_{t} l(x,dt)\big)m(dx),
\]
where the measure $l(x,dt)$ is induced by the local time $l(x,t)$ for the process $Y$ and $m(dx)$ is the corresponding reference measure; see also \cite{EISENBAUM20091401,Eisenbaum2017} and references therein.

It is shown in \cite[Thm.3.1]{fitzsimmons2014}  that the ``loop soup local time''
$\widehat{\mathcal{L}}_{\alpha}:=2\sum_{\epsilon\in\mathcal{L}_{\alpha}} l(\cdot,\infty)(\epsilon)$
 is an $\alpha$-permanental vector with kernel $g(x,y)$.
A version of the identity \eqref{eqn:3.3.2},
called the isomorphism theorem,
for the $\alpha$-permanental process
via loop soup can be found in \cite{fitzsimmons2014}.

 By the exponential formula for Poisson point processes,
 see for example \cite[O.5]{Bertoin96:book}, we obtain an alternative representation of \eqref{eqn:3.3.d}.
\begin{equation}\label{loop_soup}
\begin{split}
&\Em\Big(\exp\Big(-\sum_{\epsilon\in\mathcal{L}_{\alpha}}
\big(\sum_{j=1}^{n}p_{j}l(a_{j},\infty)(\epsilon)\big)\Big)\Big)\\
=&\exp\Big(-\alpha
\int_{\mathscr{E}}\big(1-\exp
(\sum_{j=1}^{n}p_{j}l(a_{j},\infty)(\epsilon))\big) \mu(d\epsilon)\Big).
\end{split}
\end{equation}
Note that for the permanental process $\xi$ in \eqref{eqn:3.3.d},
$\alpha=\frac{1}{\beta}=\frac{1}{2}$ and $m(dx)=1_{\{x\in(c,b)\}}\,dx$.

To reconcile \eqref{eqn:3.3.d} and \eqref{loop_soup} for the process $Y$ defined in \eqref{eqn:3.3.a},
using the results in this section we proceed  to show that
\begin{equation}\label{eqn:3.3.e}
\mu\big(1- e^{-\sum_{j=1}^{n}p_{j}l(a_{j},\infty)}\big)
=\ln\Big(\det\big(\mathbf{I}+\bl\mathbf{G}\big)\Big)
=\ln\Big(\frac{\wnqp(b,c)}{W^{(q)}(b-c)}\Big).
\end{equation}

\newcommand{\wlqp}{W^{(\lambda+q;p_{1},\cdots,p_{n})}_{(a_{1},\cdots,a_{n})}}

Following the idea of deriving Lemma 2.1 in \cite{fitzsimmons2014},
\[
\mu(F)=
\int_{\mathcal{E}}\int_{0}^{\infty}
\Pm_{x}\big( \int_{0}^{\infty} e^{-\lambda t} F\circ k_{t} l(x,dt)\big) d\lambda m(dx).
\]
We first notice that, for any $a\in(c,b)$ and $\lambda>0$
and for function
$$F(\epsilon)=1-e^{-\sum_{j=1}^{n}p_{j}l(a_{j},\infty)(\epsilon)}, \,\, \epsilon\in\mathscr{E},$$
\[
\Pm_{a}\Big(\int_{0}^{\infty} e^{-\lambda t} F\circ k_{t} l(a,dt)\Big)=
\Em_{a}\Big(\int_{0}^{\infty} e^{-\lambda t}
\big(1- e^{-\sum_{j=1}^{n}p_{j}l(a_{j},t)}\big)\mathbf{1}_{\{t<\tau_{b}^{+}\wedge\tau_{c}^{-}\wedge e_{q}\}}
l(a,dt)\Big).
\]
By the change of variable and the Fubini theorem, the right hand side of the above equation is equal to
\begin{align*}
&\ \int_{0}^{\infty}
\Em_{a}\Big(e^{-(\lambda+q) l^{-1}(a,s)}\big(1-e^{-\sum_{j=1}^{n}p_{j}l(a_{j},l^{-1}(a,s))}\big);
l^{-1}(a,s)<\tau_{b}^{+}\wedge\tau_{c}^{-}\Big)\,ds\\
=&\
\frac{W^{(q+\lambda)}(b-a)W^{(q+\lambda)}(a-c)}{W^{(q+\lambda)}(b-c)}
-\frac{\wlqp(b,a)\wlqp(a,c)}{\wlqp(b,c)}
\end{align*}
by  Theorem  \ref{cor:6}.
From the scale function identity
\[
W^{(q)}(x)-W^{(p)}(x)=(q-p)W^{(q)}*W^{(p)}(x)
\quad \text{for $p,q,x>0$,}
\]
where $*$ denotes the convolution operator,
and the fact that $W^{(q)}$ is analytic in $q$,
we have
$$W^{(q)}*W^{(q)}(x)=\frac{\partial}{\partial q}W^{(q)}(x).$$
Then
\[
 \int_{c}^{b}
\frac{W^{(q+\lambda)}(b-a)W^{(q+\lambda)}(a-c)}{W^{(q+\lambda)}(b-c)}\,da
= \frac{\frac{\partial}{\partial q}W^{(q+\lambda)}(b-c)}{W^{(\lambda+q)}(b-c)}
= \frac{\partial}{\partial \lambda} \Big(\ln W^{(\lambda+q)}(b-c)\Big).
\]
A similar relation could also be found for $\wlqp(b,c)$. Therefore, for any  $A>0$, we have
\begin{align*}
&\ \int_{0}^{A}\,d\lambda \int_{c}^{b}\Em_{a}\Big(\int_{0}^{\infty} e^{-\lambda t}
\big(1- e^{-\sum_{j=1}^{n}p_{j}l(a_{j},t)}\big)\mathbf{1}_{\{t<\tau_{b}^{+}\wedge\tau_{c}^{-}\wedge e_{q}\}}
l(a,dt)\Big)\,da\\
=&\ \int_{0}^{A}\,d\lambda
\Big(\frac{\partial}{\partial \lambda} \big(\ln W^{(\lambda+q)}(b-c)\big)
- \frac{\partial}{\partial \lambda} \big(\ln \wlqp(b-c)\big)\Big)\\
=&\ \ln\Big(\frac{\wnqp(b,c)}{W^{(q)}(b-c)}\Big)
-\ln\Big(\frac{W^{(A+q;p_{1},\cdots,p_{n})}_{(a_{1},\cdots,a_{n})}(b,c)}{W^{(A+q)}(b-c)}\Big).
\end{align*}
On the other hand, by Corollary \ref{cor:4} we have for any $a\in(c,b)$,
\[
-\log\Big(\Em_{a}\big(e^{-ql^{-1}(a,1)}; l^{-1}(a,1)<\tau_{b}^{+}\wedge\tau_{c}^{-}\big)\Big)
=\frac{W^{(q)}(b-c)}{W^{(q)}(b-a)W^{(q)}(a-c)}
\underset{q\to\infty}{\longrightarrow}\infty.
\]
It is not hard to check that
\[\D \lim_{A\to\infty}
\Big(\frac{W^{(A+q;p_{1},\cdots,p_{n})}_{(a_{1},\cdots,a_{n})}(b,c)}{W^{(A+q)}(b-c)}\Big)=1,\]
which gives \eqref{eqn:3.3.e}.

\subsection{Examples}\label{sec:4}
Considering a particular case of standard Brownian motion: $X_{t}= B_{t}$,
one can find many results on its local times in \cite{Borodin2002:book}.
For this example, $\psi(s)=\frac{1}{2}s^{2}$, $W(x)=2x$, $W^{(q)}(x)=\sqrt{\frac{2}{q}}\sinh(\sqrt{2q}x)$ and $Z^{(q)}(x)=\cosh(\sqrt{2q}x)$ for $x\geq 0$. Then
\begin{align}
W^{(q)}(x,y)=&\ W^{(q)}(x-y)= \sqrt{\frac{2}{q}}\sinh(\sqrt{2q}(x-y)^{+}),\label{eqn:exam:1}\\
\wqap(x,y)=&\ \sqrt{\frac{2}{q}}\sinh(\sqrt{2q}(x-y)^{+})\non\\
&\quad+ \frac{2p}{q}\sinh(\sqrt{2q}(x-a)^{+})\sinh(\sqrt{2q}(a-y)^{+}).\label{eqn:exam:8}
\end{align}
Given $b\geq a\geq c$,
inverting the Laplace transform in \eqref{UB:e}, we have
\begin{align*}
&\ \frac{1}{dt} \Pm_{a}(l(a,e_{q})\in\,dt, e_{q}<\tau_{b}^{+}\wedge\tau_{c}^{-})
 \exp\Big(\frac{W^{(q)}(b-c) t}{W^{(q)}(b-a)W^{(q)}(a-c)}\Big)
\non\\
=&\ \frac{W^{(q)}(a-c)(Z^{(q)}(b-c)-1)- W^{(q)}(b-c)(Z^{(q)}(a-c)-1)}{W^{(q)}(b-a)W^{(q)}(a-c)}\non\\
=&\ \sqrt{\frac{q}{2}}
\frac{\sinh(\sqrt{2q}(b-c))- \sinh(\sqrt{2q}(a-c))- \sinh(\sqrt{2q}(b-a))}
{\sinh(\sqrt{2q}(b-a))\sinh(\sqrt{2q}(a-c))}\\
=&\ \sqrt{\frac{q}{2}}
\frac{2\sinh(\sqrt{q/2}(b-a)) (\cosh(\sqrt{q/2}(b+a-2c))- \cosh(\sqrt{q/2}(b-a)))}
{\sinh(\sqrt{2q}(b-a))\sinh(\sqrt{2q}(a-c))}\\
=&\ \sqrt{\frac{q}{2}}
\frac{2\sinh(\sqrt{q/2}(a-c)) \sinh(\sqrt{q/2}(b-c))}
{\cosh(\sqrt{q/2}(b-a))\sinh(\sqrt{2q}(a-c))}\\
=& \sqrt{\frac{q}{2}} \frac{\sinh(\sqrt{q/2}(b-c))}{\cosh(\sqrt{q/2}(b-a)) \cosh(\sqrt{q/2}(a-c))}.
\end{align*}

One can check that the expression above coincides with Table 1.25.2
on P185 of \cite{Borodin2002:book}. Similarly, we can recover Tables
3.3.7 on P215 and 4.15.4 on P236 of \cite{Borodin2002:book} by making use of
\eqref{eqn:exam:1} and \eqref{eqn:exam:8}.
And with some more work, one can also verify 3.18.5, 4.16.1 and 4.18.1 of  \cite{Borodin2002:book}.

For their joint distributions, for $c\leq v\leq 0\leq u\leq b$
we have from Corollary \ref{cor:6} that
\begin{align*}
&\ -\log\Big(\Em\Big(e^{- p l(u,l^{-1}(0,1))- q l(v,l^{-1}(0,1))}; l^{-1}(0,1)<\tau_{b}^{+}\wedge\tau_{c}^{-}\Big)\Big)
=\frac{\mathsf{W}^{(p,q)}_{(u,v)}(b,c)}{\mathsf{W}^{(p,q)}_{(u,v)}(b,0)\mathsf{W}^{(p,q)}_{(u,v)}(0,c)}\\
=&\ \frac{2(b-c)+ 4p(b-u)(u-c)+ 4q (b-v)(v-c)+ 8pq(b-u)(u-v)(v-c)}
{(2b+4p(b-u)u) (-2c+4q(-v)(v-c))}\\
&\underset{c\to-\infty, b\to\infty }{\longrightarrow} \frac{4p+4q+8pq(u-v)}{4(1+2pu) (1-2qv)}= \frac{p}{1+2pv}+ \frac{q}{1-2qv},
\end{align*}
which results in the independence of $l(v,l^{-1}(0,1))$ and
$l(u,l^{-1}(0,1))$ as expected by the Ray-Knight theorem.

 For linear Brownian motion $X_{t}=\mu t+ B_{t}$, $\psi(s)=\frac{1}{2}s^{2}+\mu s$ and
 \[W(x)=\frac{1}{\mu}(1-e^{-2\mu x})=\frac{2}{\mu} e^{-\mu x}\sinh(\mu x).\]
 By Corollary \ref{cor:4} we have
\begin{align*}
\Pm_{a}\Big(l(a,\tau_{b}^{+})\geq t\Big| \tau_{b}^{+}<\tau_{c}^{-}\Big)=&\ \exp\Big(-t \frac{W(b-c)}{W(b-a)W(a-c)}\Big)\\
=&\ \exp\Big(-\frac{\mu t}{2} \frac{\sinh(\mu(b-c))}{\sinh(\mu(b-a))\sinh(\mu(a-c))}\Big),
\end{align*}
which also coincides with the result on Page 93 of
\cite{Borodin2002:book}.

\section{Proofs of the main results}\label{sec:proof}

\begin{proof}[Proof of Lemma \ref{lem:hitting}] As observed in \cite{Ivanovs2012:occupationdensity:map} that, due to the absence of positive jumps,
$$\{\tau_{a}^{-}<\tau_{b}^{+}<\infty\}=\{\tau^{\{a\}}<\tau_{b}^{+}<
\infty\} \,\,\,\,\,
\Pm_x \, \text{ a.s.}$$
Therefore, for $x\in[c,b]$
\begin{align*}
&\ \Em_{x}\big(e^{- q\tau_{b}^{+}}; \tau_{b}^{+}<\tau_{c}^{-}\big)\\
=&\ \Em_{x}\big(e^{- q\tau_{b}^{+}}; \tau_{b}^{+}<\tau_{c}^{-}, \tau_{b}^{+}<\tau_{a}^{-}\big)+
\Em_{x}\big(e^{-q \tau_{b}^{+}}; \tau_{b}^{+}<\tau_{c}^{-}, \tau_{a}^{-}<\tau_{b}^{+}\big)\\
=&\ \Em_{x}\big(e^{-q \tau_{b}^{+}}; \tau_{b}^{+}<\tau_{a}^{-}\big)+
\Em_{x}\big(e^{- q\tau_{b}^{+}}; \tau^{\{a\}}<\tau_{b}^{+}<\tau_{c}^{-}\big)\\
=&\ \Em_{x}\big(e^{-q \tau_{b}^{+}}; \tau_{b}^{+}<\tau_{a}^{-}\big)+ \Em_{x}\big(e^{- q\tau^{\{a\}}}; \tau^{\{a\}}<\tau_{b}^{+}\wedge\tau_{c}^{-}\big) \Em_{a}\big(e^{-q\tau_{b}^{+}}; \tau_{b}^{+}<\tau_{c}^{-}\big).
\end{align*}
where the Markov property is applied in the last identity.
{Solving the above equation for
{$\Em_{x}\big(e^{- q\tau^{\{a\}}}; \tau^{\{a\}}<\tau_{b}^{+}\wedge\tau_{c}^{-}\big)$}
and applying \eqref{one-sided}, the result follows.}
\end{proof}

\subsection{Weighted occupation times}\label{sec:5.1}
The proof of our main results rely on the study of the associated occupation times
and the following lemmas concerning the functions introduced.
For a locally bounded nonnegative function $\omega$ on $\mathbb{R}$,
the $\omega$-weighted occupation time up to time $t$ is defined by
\begin{equation}\label{defn:L}
L(t):=\int_{0}^{t} \omega(X_{s})\,ds\quad\text{for all $t\geq0$.}
\end{equation}
Let $\wo(\cdot,\cdot)$ and $\zo(\cdot,\cdot)$ be, respectively,
the unique locally bounded solutions to integral equations
\begin{equation}
\wo(x,y)= W(x-y)+ \int_{y}^{x}W(x-z)\omega(z)\wo(z,y)\,dz\label{defn:Ww}
\end{equation}
and
\begin{equation}
 \zo(x,y)= 1+ \int_{y}^{x}W(x-z)\omega(z)\zo(z,y)\,dz\label{defn:Zw}
\end{equation}
for $x, y\in\mathbb{R}$.
Their existence as well as uniqueness are assured by the following Lemma
which can be derived from \cite[Lemma 2.1]{BoLi:sub1} with a shifting argument, where $\wo$ and $\zo$
are called $\omega$-scale functions in the paper.
In particular, for $x<y$, $\wo(x,y)=W(x-y)=0$ and the integral in \eqref{defn:Ww} can be extended to $\mathbb{R}$.

\begin{lem}\label{lem:equation} Let $h(\cdot,\cdot)$ and $\omega(\cdot)\geq 0$ be locally bounded functions defined on $\mathbb{R}^2$ and $\mathbb{R}$, respectively.
The integral equation
\begin{equation}
H^{(\omega)}(x,y)=h(x,y)+ \int_{y}^{x} W(x-z) \omega(z) H^{(\omega)}(z,y)\,dz \label{lem:eqn:1}, \quad x, y\in\mathbb{R},
\end{equation}
admits a unique locally bounded solution $H^{(\omega)}$ on $\mathbb{R}^2$ with
$H^{(\omega)}(x,y)=h(x,y)$ for $x\leq y$.
\end{lem}

The following results on weighted occupation times of SNLPs can be found in \cite{BoLi:sub1}.
\begin{prop}[$\omega$-fluctuation identities]\label{prop:levy:omega} Given $b>c$, we have for $x\in[c,b]$
\[
\Em_{x}\big(e^{-L(\tau_{b}^{+})}; \tau_{b}^{+}<\tau_{c}^{-}\big)
= \frac{\wo(x,c)}{\wo(b,c)}
\]
and
\[
\Em_{x}\big(e^{-L(\tau_{c}^{-})}; \tau_{c}^{-}<\tau_{b}^{+}\big)
= \zo(x,c)-\frac{\wo(x,c)}{\wo(b,c)} \zo(b,c).
\]
The $\omega$-resolvent of $X$ killed at exiting $[c,b]$ is given by, for $x,y\in(c,b)$,
\begin{align}
\uo(x,dy):=&\ \int_0^{\infty} \Em_{x}\big( e^{-L(t)}; X_t\in dy, t<\tau_{b}^{+}\wedge\tau_{c}^{-}\big) dt\non\\
=&\ \Big(\frac{\wo(x,c)}{\wo(b,c)} \wo(b,y)- \wo(x,y)\Big) dy.
\end{align}
\end{prop}

More properties of the $\omega$-scale functions are discovered in this paper.



\begin{lem}\label{W-identity}
For any $x, y\in\mathbb{R}$, we have
\begin{equation}\label{defn:Ww:2}
\wo(x,y)=W(x-y)+ \int_{y}^{x}\wo(x,z)\omega(z)W(z-y)\,dz.
\end{equation}
\end{lem}
\begin{proof}[Proof of Lemma \ref{W-identity}]
Denoting by $g(x,y)$ the right hand side of \eqref{defn:Ww:2},
plugging it into \eqref{defn:Ww}, we have for $x,y\in\mathbb{R}$
\begin{align*}
g(x,y):=&\ W(x-y)\\
&\ +\int_{y}^{x}\Big(W(x-z)+\int_{z}^{x}W(x-u)\omega(u)\wo(u,z)du\Big)\omega(z)W(z-y)\,dz\\
=&\ W(x-y)+ \int_{y}^{x}W(x-z)\omega(z)W(z-y)\,dz\\
&\ + \int_{y}^{x} W(x-u)\omega(u) \Big(\int_{y}^{u}\wo(u,z)\omega(z)W(z-y)\,dz\Big)\,du\\
=&\ W(x-y)+ \int_{y}^{x}W(x-u)\omega(u)g(u,y)\,du.
\end{align*}
We thus have $g(x,y)=\wo(x,y)$ following the uniqueness of  the solution to \eqref{lem:eqn:1}.
\end{proof}

\begin{lem}\label{lem:equation:2}
Let $(\wo,\zo)$ and
$(W^{(\omega_{1})},Z^{(\omega_{1})})$ be the
scale functions for weight functions
$\omega(\cdot)\geq 0$ and $\omega_{1}(\cdot)\geq 0$, respectively. Then
for $x,y\in\mathbb{R}$
\begin{equation}
W^{(\omega_{1})}(x,y)-\wo(x,y)= \int_{y}^{x}\wo(x,z)(\omega_{1}(z)-\omega(z))W^{(\omega_{1})}(z,y)\,dz\label{lem:eqn:2}
\end{equation}
and
\begin{equation}
Z^{(\omega_{1})}(x,y)-\zo(x,y)= \int_{y}^{x}\wo(x,z)(\omega_{1}(z)-\omega(z))Z^{(\omega_{1})}(z,y)\,dz.\label{lem:eqn:3}
\end{equation}
\end{lem}
\begin{proof}[Proof of Lemma \ref{lem:equation:2}]
{Here we only need to consider the nontrival case $x\geq y$.}
Applying identity \eqref{defn:Ww:2} to $\wo(x,z)$ and identity
\eqref{defn:Ww} to $W^{(\omega_{1})}(z,y)$ twice in the following integrations, we have
{for $x\geq y$}
\begin{align*}
&\ \int_{y}^{x}\wo(x,z)\omega_{1}(z)W^{(\omega_{1})}(z,y)\,dz\\
=&\ \int_{y}^{x}W(x-z)\omega_{1}(z)W^{(\omega_{1})}(z,y)\,dz\\
&\ + \iint_{x\geq u\geq z\geq y} \wo(x,u)\omega(u)W(u-z)\omega_{1}(z)W^{(\omega_{1})}(z,y)\,dz\,du\\
=&\ W^{(\omega_{1})}(x,y)-W(x-y)+ \int_{y}^{x}\wo(x,u)\omega(u)(W^{(\omega_{1})}(u,y)-W(u-y))\,du\\
=&\ W^{(\omega_{1})}(x,y)-\wo(x,y)+ \int_{y}^{x}\wo(x,u)\omega(u)W^{(\omega_{1})}(u,y)\,du.
\end{align*}
 \eqref{lem:eqn:2} is thus proved. Identity \eqref{lem:eqn:3} can be proved similarly.
\end{proof}

If $\omega(\cdot)\equiv q$, then $\wo(x,y)=W^{(q)}(x-y)$.
If $\omega_{1}(\cdot)\equiv p$ and $\omega(\cdot)\equiv q$, the identities in the above Lemma are reduced to the classical scale function identities, which were investigated and used to define auxiliary functions in \cite{Loeffen2014:occupationtime:levy} as follows:
for $b>a>0$ in their paper, $p,q\geq 0$ and $x\in\mathbb{R}$
\begin{align*}
W^{(p,q)}_{(a)}(x):=&\ W^{(q)}(x)+ (p-q) \int_{0}^{a}W^{(q)}(x-z)W^{(p)}(z)\,dz,\\
W^{(p,q,p)}_{(a,b)}(x):=&\ W^{(p,q)}_{(a)}(x)+ (p-q) \int_{b}^{x} W^{(p)}(x-z)W^{(q,p)}_{(a)}(z)\,dz.
\end{align*}
Actually, $W^{(p,q,p)}_{(a,b)}(x)=\wo(x,0)$ in our paper for
the weight function $\omega(z)=p+ (q-p) \mathbf{1}_{z\in(a,b)}$
 for $ z\in\mathbb{R}$,
and we prove that
$W^{(p,p+\frac{q}{2\varepsilon},p)}_{(a-\varepsilon,a+\varepsilon)}(x,y)
\underset{\varepsilon\to0+}{\longrightarrow} W^{(p;q)}_{(a)}(x,y)$ for all $a,x,y\in\mathbb{R}$.
More specifically, for the study of local times as defined in \eqref{defn:localtime}, for any $a\in\mathbb{R}$ and $p>0$,
we  consider the  approximate delta functions $\omega_{\varepsilon}(x):=\frac{p}{2\varepsilon}
\boldsymbol{1}_{\{|x-a|\leq\varepsilon\}}$ for $\varepsilon>0$.

Recall that we assume $W(0)=0$ throughout the paper.
\begin{lem}\label{lem:con:omega}
For $c,a\in\mathbb{R}$ with $c<a$, let $h(\cdot)\geq 0$ be a function continuous at $a$,
$\wo(\cdot,\cdot)$ be the $\omega$-scale function with respect to
a locally bounded function $\omega(\cdot)\geq 0$,
and $\hoe(\cdot)$ be a function satisfying
\begin{equation}\label{eqn:lem:con}
\hoe(x)=h(x)+ \int_{c}^{x} \wo(x,z) \omega_{\varepsilon}(z) \hoe(z)\,dz,
\quad\text{for all $x\in\mathbb{R}$,}
\end{equation}
where $\omega_{\varepsilon}(x):=\frac{p}{2\varepsilon}
\boldsymbol{1}_{\{|x-a|\leq\varepsilon\}}, \, x\in\mathbb{R}$, for some $p\geq 0$.
We have
\[
\hoe(x)\underset{\varepsilon\to0}{\longrightarrow}
\hoap(x):=h(x)+ p \wo(x,a) h(a),
\quad\text{for all $x\in\mathbb{R}$.}
\]
\end{lem}

\begin{proof}[Proof of Lemma \ref{lem:con:omega}]
Not that for $x\in\mathbb{R}$,
\[
\hoe(x)=
\big(h(x)+\int_{c}^{x}W(x-y)\omega(y)h(y)\,dy\big)
+ \int_{c}^{x}W(x-y)\big(\omega(y)+\omega_{\varepsilon}(y)\big)\hoe(y)\,dy.
\]
{Then by Lemma  \ref{lem:equation} $\hoe$ is 
the unique locally bounded solution to \eqref{eqn:lem:con}.}

It is not hard to see from \eqref{defn:Ww} and \eqref{defn:Ww:2} that
$\wo(x,\cdot)$ and $\wo(\cdot,y)$ are monotone for $x,y\in\mathbb{R}$, respectively,
and $\wo(\cdot,\cdot)$ is continuous on $\mathbb{R}^{2}$.

By Lemma \ref{lem:equation} we only need to focus on the nontrivial case $x>c$
and mostly the existence of $\D \lim_{\varepsilon\to0+}\hoe(x)$.
It holds that $\hoe(y)-h(y)\geq 0$ for all $y\in\mathbb{R}$, and
$\hoe(\cdot)$ is a locally bounded.
For small $\varepsilon>0$ and $x\in\mathbb{R}$ define
\[
k_{\varepsilon}(x):=\int_{c}^{x}\wo(x,z)\omega_{\varepsilon}(z)h(z)\,dz
\quad\text{and}\quad
\overline{H}(\varepsilon):=\int_{\mathbb{R}}\omega_{\varepsilon}(z)\big(\hoe(z)-h(z)\big)\,dz.
\]
By the continuity and the monotonicity of functions involved,
we have for $x>c$
\[
k_{\varepsilon}(x)\leq p \wo(x,a-\varepsilon) \sup_{\{|z-a|\leq\varepsilon\}}h(z)
\quad\text{and}\quad
k_{\varepsilon}(x)\underset{\varepsilon\to0}{\longrightarrow}  p \wo(x,a)h(a).
\]
Notice that the convergence above holds for $x\leq a$ only under the assumption $W(0)=0$.
Applying \eqref{eqn:lem:con} we further have for $x>c$
\begin{align*}
\hoe(x)-&\ h(x)= k_{\varepsilon}(x)
+ \int_{c}^{x}\wo(x,z)\omega_{\varepsilon}(z)\big(\hoe(z)-h(z)\big)\,dz\non\\
\leq&\ k_{\varepsilon}(x)+ \wo(x,a-\varepsilon) \times\overline{H}(\varepsilon)\\
\leq&\ \wo(x,a-\varepsilon)\times
\Big(p \sup_{\{|z-a|\leq \varepsilon\}}h(z)+ \overline{H}(\varepsilon)\Big).
\end{align*}
Multiplying  both sides of equation above by $\omega_{\varepsilon}(x)$
and integrating in $x$ over $\mathbb{R}$,
since $\omega_{\varepsilon}(x)=0$ for $x\notin[a-\varepsilon,a+\varepsilon]$,
it follows that
\[
\overline{H}(\varepsilon)\leq p \wo(a+\varepsilon,a-\varepsilon)\times
\Big(p \sup_{\{|z-a|\leq \varepsilon\}}h(z)+ \overline{H}(\varepsilon)\Big).
\]
{Solving the inequality above for $\overline{H}(\varepsilon)$},
   since $W(0)=0$, we have
\[
0\leq \overline{H}(\varepsilon)\leq\frac{p^{2} \wo(a+\varepsilon,a-\varepsilon)}{1-p \wo(a+\varepsilon,a-\varepsilon)}
\Big(\sup_{\{|z-a|\leq\varepsilon\}}h(z)\Big)\underset{\varepsilon\to0}{\longrightarrow} 0,
\]
which completes the proof.
\end{proof}

\begin{rmk}\label{rmk:limit:WZe}
A consequence of Lemmas \ref{lem:equation:2} and \ref{lem:con:omega} is that,
for $a,c,x,y\in\mathbb{R}$ with $a>c$,
since $\wo(x,y)$ and $\zo(x,c)$ is continuous at $x=a$,
\begin{equation}\label{limit_a}
\left\{\begin{split}
\woe(x,y)\underset{\varepsilon\to 0}{\longrightarrow}&\ \woap(x,y):=\wo(x,y)+p\wo(x,a)\wo(a,y),\\
\zoe(x,c)\underset{\varepsilon\to 0}{\longrightarrow}&\ \zoap(x,c):=\zo(x,c)+p \wo(x,a)\zo(a,c),
\end{split}\right.
\end{equation}
since we can rewrite the equation as for $x,y\in\mathbb{R}$
\[
\woe(x,y)=\wo(x,y)+ \int_{c}^{x}\wo(x,z)\omega_{\varepsilon}(z)\woe(z,y)\,dz
\]
In particular, letting $\omega(\cdot)\equiv q$  we recover the previous definitions in \eqref{eqn:1}.
\end{rmk}

\subsection{Proofs of Theorems \ref{thm:loc:q}, \ref{cor:5} and \ref{cor:6}}
Given the lemmas in Section \ref{sec:5.1}, we are ready to prove our main results.
\begin{proof}[Proof of Theorem \ref{thm:loc:q}]
The desired formulas can be derived directly from Proposition
\ref{prop:levy:omega}, Remark \ref{rmk:limit:WZe} and the
approximating identities. Put
$\omega_{\varepsilon}(x):=\frac{p}{2\varepsilon}
\boldsymbol{1}_{\{|x-a|\leq \varepsilon\}}$ for $x\in\mathbb{R}$. We have
\begin{align*}
\Em_{x}\big(e^{-q\tau_{b}^{+}-p l(a,\tau_{b}^{+})}; \tau_{b}^{+}<\tau_{c}^{-}\big)
=&\ \lim_{\varepsilon\to0+}
\Em_{x}\big(e^{-\int_{0}^{\tau_{b}^{+}}(q+\omega_{\varepsilon}(X_{s}))\,ds}; \tau_{b}^{+}<\tau_{c}^{-}\big)
\end{align*}
and
\begin{align*}
\Em_{x}\big(e^{-q \tau_{c}^{-}-p l(a,\tau_{c}^{-})}; \tau_{c}^{-}<\tau_{b}^{+}\big)
=&\ \lim_{\varepsilon\to0+}
\Em_{x}\big(e^{-\int_{0}^{\tau_{c}^{-}}(q+\omega_{\varepsilon}(X_{s}))\,ds}; \tau_{c}^{-}<\tau_{b}^{+}\big).
\end{align*}

Since $e_{q}$ is exponentially distributed variable and independent of $X$, for any nonnegative and bounded function $f$ we have
\begin{align*}
&\ \Em_{x}\big(e^{-\int_{0}^{e_{q}}\omega_{\varepsilon}(X_{s})\,ds}f(X_{e_{q}}); e_{q}<\tau_{b}^{+}\wedge\tau_{c}^{-}\big)\\
=&\ q \int_{0}^{\infty} e^{-qt} \Em_{x}\big(e^{-\int_{0}^{t}\omega_{\varepsilon}(X_{s})\,ds}f(X_{t}); t<\tau_{b}^{+}\wedge\tau_{c}^{-}\big)\,dt\\
=&\ q \int_{0}^{\infty} \Em_{x}\big(e^{-\int_{0}^{t}(q+\omega_{\varepsilon}(X_{s}))\,ds} f(X_{t}); t<\tau_{b}^{+}\wedge\tau_{c}^{-}\big)\,dt\\
=&\ q \int_{c}^{b}\Big(\frac{\wqe(x,c)}{\wqe(b,c)}\wqe(b,y)-\wqe(x,y)\Big)f(y)\,dy.
\end{align*}
The proof is completed by taking a limit.
\end{proof}

\begin{proof}[Proof of Theorem \ref{cor:5}]
Theorem \ref{cor:5} is proved similarly, but by a sequence of approximate delta functions.
Let $\{p_{k}>0, a_{k}\in\mathbb{R},k\geq1\}$ as given in \eqref{eqn:2}.
Let $$\omega_{\varepsilon_{j}}(x):=\frac{p_{j}}{2\varepsilon_{j}}
\boldsymbol{1}_{\{|x-a_{j}|\leq \varepsilon_{j}\}} \,\,\, \text{and} \,\,\,
\omega_{k}(x):=\sum_{j=1}^{k}\omega_{\varepsilon_{j}}(x)
\quad\text{for $x\in\mathbb{R}$}$$ and
$W^{(q+\omega_{k})}$ be the scale function with respect to
$q+\omega_{k}$. It follows from Proposition \ref{prop:levy:omega}
that
\[
\Em_{x}\Big(\exp\big(-q \tau_{b}^{+}- \sum_{j=1}^{k} \frac{p_{j}}{2\varepsilon_{j}} \int_{0}^{\tau_{b}^{+}} \boldsymbol{1}_{\{|X_{t}-a_{j}|\leq \varepsilon_{j}\}}\,dt \big); \tau_{b}^{+}<\tau_{c}^{-}\Big)
=\frac{W^{(q+\omega_{k})}(x,c)}{W^{(q+\omega_{k})}(b,c)}.
\]

Applying Lemmas \ref{lem:equation:2} and \ref{lem:con:omega}
successively, we have for all $x,y\in\mathbb{R}$,
\begin{align*}
W^{(q+\omega_{k})}(x,y)=&\ W^{(q+\omega_{k-1})}(x,y)+ \int_{y}^{x}
W^{(q+\omega_{k-1})}(x,z)(\omega_{k}(z)-\omega_{k-1}(z))W^{(q+\omega_{k})}(z,y)\,dz\\
=&\ W^{(q+\omega_{k-1})}(x,y)+
\int_{y}^{x} W^{(q+\omega_{k-1})}(x,z) \omega_{\varepsilon_{k}}(z) W^{(q+\omega_{k})}(z,y)\,dz\\
\underset{\varepsilon_{k}\to0}{\longrightarrow}& W^{(q+\omega_{k-1})}(x,y)+ p_{k} W^{(q+\omega_{k-1})}(x,a_{k}) W^{(q+\omega_{k-1})}(a_{k},y)\\
=&\ W^{(q+\omega_{k-1})}(x,y)+ p_{k} W^{(q)}(x- a_{k}) W^{(q+\omega_{k-1})}(a_{k},y),
\end{align*}
where the fact that $W^{(q+\omega_{k-1})}(x,a_{k})=W^{(q)}(x-a_{k})$ for all $x\in\mathbb{R}$ is needed {for the last equality.}
Applying the above result repeatedly, we have
\begin{align*}
W^{(q+\omega_{1})}(x,y)\underset{\varepsilon_{1}\to0}{\longrightarrow}&\ W^{(q)}(x-y)+ p_{1} W^{(q)}(x-a_{1}) W^{(q)}(a_{1}-y)= \mathsf{W}^{(q;p_{1})}_{(a_{1})}(x,y)
\end{align*}
and
\begin{align*}
W^{(q+\omega_{2})}(x,y)\underset{\varepsilon_{2}\to0}{\longrightarrow} &\
 W^{(q+\omega_{1})}(x,y)+ p_{2} W^{(q)}(x-a_{2}) W^{(q+\omega_{1})}(a_{2},y)\\
\underset{\varepsilon_{1}\to0}{\longrightarrow}&\ \mathsf{W}^{(q;p_{1})}_{(a_{1})}(x,y)+ p_{2} W^{(q)}(x-a_{2})
\mathsf{W}^{(q;p_{1})}_{(a_{1})}(a_{2},y)
=\mathsf{W}^{(q;p_{1},p_{2})}_{(a_{1}, a_{2})}(x,y),
\end{align*}
Then the identity \eqref{joint_a}
for $k=1$ follows by taking a limit.
The identity \eqref{joint_a} for $k\geq2$ can be shown similarly.
So does the resolvent measure.


Similarly, applying Lemmas \ref{lem:equation:2} and \ref{lem:con:omega} one has, for $v\in\mathbb{R}$ with $a_{k}>v$,
\[
Z^{(q+\omega_{k})}(x,v)
\underset{\varepsilon_{k}\to0}{\longrightarrow}
Z^{(q+\omega_{k-1})}(x,v)+ p_{k} W^{(q)}(x-a_{k}) Z^{(q+\omega_{k-1})}(a_{k},v).
\]
In particular, the approximation holds for $v\equiv c$ and for all $k$. Therefore,
\begin{align*}
Z^{(q+\omega_{1})}(x,c)\underset{\varepsilon_{1}\to0}{\longrightarrow}&\
Z^{(q)}(x-c)+ p_{1} W^{(q)}(x-a_{1}) Z^{(q)}(a_{1}-c)= \mathsf{Z}^{(q;p_{1})}_{(a_{1})}(x,c)
\end{align*}
and
\begin{align*}
Z^{(q+\omega_{2})}(x,c)\underset{\varepsilon_{2}\to0}{\longrightarrow} &\
Z^{(q+\omega_{1})}(x,y)+ p_{2} W^{(q)}(x-a_{2}) Z^{(q+\omega_{1})}(a_{2},c)\\
\underset{\varepsilon_{1}\to0}{\longrightarrow}
&\ \mathsf{Z}^{(q;p_{1})}_{(a_{1})}(x,c)+ p_{2} W^{(q)}(x-a_{2})
\mathsf{Z}^{(q;p_{1})}_{(a_{1})}(a_{2},c)
= \mathsf{Z}^{(q;p_{1},p_{2})}_{(a_{1}, a_{2})}(x,c),
\end{align*}
which eventually give \eqref{joint_b}.
This completes the proof.
\end{proof}

To derive the local times at inverse local time,
we first consider the $\omega$-weight occupation time at $l^{-1}(a,t)$,
that is, $L(l^{-1}(a,t))$ for $t>0$.
\begin{prop}[Occupation time at inverse local time]\label{thm:occu} For $b>a>c$, let $L(\cdot)$ be the $\omega$-weighted occupation time defined in \eqref{defn:L} for some $\omega(\cdot)\geq 0$. We have
\begin{equation}
\Em_{a}\big(e^{-L(l^{-1}(a,t))}; l^{-1}(a,t)<\tau_{b}^{+}\wedge\tau_{c}^{-}\big)
= \exp\Big(-\frac{\wo(b,c) t}{\wo(b,a)\wo(a,c)}\Big).
\end{equation}
\end{prop}
\begin{proof}[Proof of Proposition \ref{thm:occu}]
For $\omega_{\varepsilon}(x):=\frac{p}{2\varepsilon} \boldsymbol{1}_{\{|x-a|\leq \varepsilon\}}$ for $x\in\mathbb{R}$, we have
\begin{align*}
&\ \Em_{a}\big(e^{-L(\tau_{b}^{+})- p l(a,\tau_{b}^{+})}; \tau_{b}^{+}<\tau_{c}^{-}\big)
=\lim_{\varepsilon\to0+}
\Em_{a}\big(e^{-L(\tau_{b}^{+})-\int_{0}^{\tau_{b}^{+}}\omega_{\varepsilon}(X_{s})\,ds};
\tau_{b}^{+}<\tau_{c}^{-}\big)\\
=&\ \lim_{\varepsilon\to0+}\frac{\woe(a,c)}{\woe(b,c)}
= \frac{\wo(a,c)}{\wo(b,c)+ p \wo(b,a)\wo(a,c)}.
\end{align*}
Inverting the Laplace transform yields
\[
\Em_{a}\big(e^{-L(\tau_{b}^{+})}; l(a,\tau_{b}^{+})> t, \tau_{b}^{+}<\tau_{c}^{-}\big)
=\frac{\wo(a,c)}{\wo(b,c)}\exp\Big(-\frac{\wo(b,c) t}{\wo(b,a)\wo(a,c)}\Big).
\]

Recalling that for every $t\geq 0$,
$X(l^{-1}(a,t))=a$ on $\{l^{-1}(a,t)<\infty\}$
and
$\{l(a,\tau_{b}^{+})>t\}=\{l^{-1}(a,t)<\tau_{b}^{+}\}$,
applying the Markov property at time $l^{-1}(a,t)$, we have
\begin{align*}
&\ \Em_{a}\big(e^{-L(\tau_{b}^{+})}; l(a,\tau_{b}^{+})> t, \tau_{b}^{+}<\tau_{c}^{-}\big)\\
=&\ \Em_{a}\big(e^{-L(\tau_{b}^{+})}; l^{-1}(a,t)<\tau_{b}^{+}, \tau_{b}^{+}<\tau_{c}^{-}\big)\\
=&\ \Em_{a}\big(e^{-L(l^{-1}(t))}; l^{-1}(a,t)<\tau_{b}^{+}\wedge\tau_{c}^{-}\big) \Em_{a}\big(e^{-L(\tau_{b}^{+})}; \tau_{b}^{+}<\tau_{c}^{-}\big).
\end{align*}
This completes the proof.
\end{proof}

\begin{proof}[Proof of Theorem \ref{cor:6}] Let $\omega_{n}(x)$ be the auxiliary function defined in the proof of Theorem \ref{cor:5}. We have from Proposition \ref{thm:occu} that
\begin{align*}
&\ \Em\Big(\exp\big(-\int_{0}^{l^{-1}(a,t)}(q+\omega_{n}(X_{s}))\,ds\big); l^{-1}(a,t)<\tau_{b}^{+}\wedge\tau_{c}^{-}\Big)\\
 =&\
 \exp\Big(\frac{-W^{(q+\omega_{n})}(b,c)t}{W^{(q+\omega_{n})}(b,a) W^{(q+\omega_{n})}(a,c)}\Big)
 \underset{\varepsilon_{k}\to0}{\longrightarrow}
 \exp\Big(\frac{-\wnqp(b,c) t}{\wnqp(b,a)\wnqp(a,c)}\Big)
\end{align*}
for $k=1,\cdots,n$,
which gives the desired expression.
\end{proof}

\subsection{Matrix expressions for $\wnqp$ and $\znqp$}
We fixed $n\in\mathbb{N}$ in this subsection.
\begin{lem} \label{lem:2}
Let $\wnqp$ and $\znqp$ be  defined as in \eqref{eqn:2}.
For $x,y\in\mathbb{R}$, we have
\begin{align}\label{eqn:wn:m1}
\wnqp(x,y)
= W^{(q)}(x-y)+ \sum_{k=1}^{n} W^{(q)}(x-a_{k}) p_{k} \wnqp(a_{k},y)
\end{align}
and
\begin{align}\label{eqn:zn:m1}
\znqp(x,c)
= Z^{(q)}(x-c)+ \sum_{k=1}^{n} W^{(q)}(x-a_{k}) p_{k} \znqp(a_{k},c).
\end{align}
\end{lem}
\begin{proof}[Proof of Lemma \ref{lem:2}]

For arbitrarily small $\eta_{j}\geq \varepsilon_{j}>0$, $j=1,\cdots,n$,
we have by applying Lemma \ref{lem:equation:2} that for $x,y\in\mathbb{R}$,
\begin{align*}
W^{(q+\omega_{n})}(x,y)=&\ W^{(q)}(x-y)+ \int_{\mathbb{R}} W^{(q)}(x-z) \omega_{n}(z) W^{(q+\omega_{n})}(z,y)\,dz\\
\leq &\ W^{(q)}(x-y)+ \sum_{j=1}^{n} p_{j} W^{(q)}(x-a_{j}+\varepsilon_{j}) W^{(q+\omega_{n})}(a_{j}+\varepsilon_{j},y)\\
\leq &\ W^{(q)}(x-y)+ \sum_{j=1}^{n} p_{j} W^{(q)}(x-a_{j}+\eta_{j}) W^{(q+\omega_{n})}(a_{j}+\eta_{j},y)
\end{align*}
where
$\omega_{n}(x)=\sum_{j=1}^{n}\frac{p_{j}}{2\varepsilon_{j}}\boldsymbol{1}_{\{|x-a_{j}|\leq \varepsilon_{j}\}}$,
and the monotonicity of $W^{(q+\omega_{n})}(\cdot, y)$ is used.
First letting $\varepsilon_{j}\to0$ and then letting $\eta_{j}\to0$ for $j=1,\cdots,n$, we have
for $x,y\in\mathbb{R}$
\begin{align*}
&\ W^{(q)}(x-y)+ \sum_{j=1}^{n} p_{j} W^{(q)}(x-a_{j}+\eta_{j}) W^{(q+\omega_{n})}(a_{j}+\eta_{j},y)\\
\underset{\varepsilon_{j}\to0}{\longrightarrow} &\ W^{(q)}(x-y)+ \sum_{j=1}^{n} p_{j} W^{(q)}(x-a_{j}+\eta_{j}) \wnqp(a_{j}+\eta_{j},y)\\
\underset{\eta_{j}\to0}{\longrightarrow} &\ W^{(q)}(x-y)+ \sum_{j=1}^{n} p_{j} W^{(q)}(x-a_{j}) \wnqp(a_{j},y)
\end{align*}
by making use of the existence results and the continuity of the functions involved.
Similarly, one can obtain the same lower bound and those for $\znqp$, which eventually
gives the desired expression.
\end{proof}

\begin{proof}[Proof of Proposition \ref{prop:1}]
{Given the matrix representations in Lemma \ref{lem:2},
taking $x=a_{i}$ for every $i=1,\cdots,n$ in \eqref{eqn:wn:m1}, we have} for $y\in\mathbb{R}$
\[
\big(\wnqp(a_{i},y)\big)
= \bb(y)+
\bs\cdot \bl \cdot\big(\wnqp(a_{i},y)\big).
\]
Since $\bl\bs$ is a strictly lower triangular matrix,
it follows that
\[
\bl\cdot\big(\wnqp(a_{i},y)\big)
=(\mathbf{I}-\bl\bs)^{-1}\cdot\bl\cdot\bb(y).
\]
Plugging it into \eqref{eqn:wn:m1}, we have for $x,y\in\mathbb{R}$,
\[
\wnqp(x,y)=W^{(q)}(x-y)+ \ba^{\mathrm{T}}(x) (\mathbf{I}-\bl\bs)^{-1}\bl\bb(y).
\]
On the other hand, making use of a matrix multiplication
\[
 \begin{pmatrix}
W^{(q)}(x-y) & \ba^{\mathrm{T}}(x)\\
-\bl \bb(y) & \mathbf{I}-\bl\bs
\end{pmatrix}\times
\begin{pmatrix}
1 & \mathbf{0}\\
(\mathbf{I}-\bl\bs)^{-1}\bl\bb(y) &\ (\mathbf{I}-\bl\bs)^{-1}
\end{pmatrix}
\]
and the fact that $(\mathbf{I}-\bl\bs)^{-1}$ is also a lower triangular matrix with entries $1$ on the diagonal, we complete the proof of \eqref{eqn:3}. Similarly, identity \eqref{eqn:4} follows from \eqref{eqn:zn:m1}.
\end{proof}

\paragraph{Acknowledgement:}
We are grateful to an anonymous referee for numerous very helpful comments and suggestions.
Bo Li thanks Concordia University where the work on this paper was carried out during his visits.


\end{document}